\documentclass[10pt, twoside, reqno]{amsart}
\usepackage{circledsteps}
\usepackage{xcolor, amstext, amsfonts, amssymb, amsbsy, latexsym}
\usepackage{enumerate}
\usepackage[T1]{fontenc}
\usepackage{xy, hhline}
\newtheorem{lemma}{Lemma}[section]
\newtheorem{theo}[lemma]{Theorem}

\newtheorem{conv}[lemma]{Convention}
\newtheorem{cor}[lemma]{Corollary}
\newtheorem{conj}[lemma]{Conjecture}
\theoremstyle{definition}
\newtheorem{defin}[lemma]{Definition}
\newtheorem{example}[lemma]{Example}
\newtheorem{remark}[lemma]{Remark}

\newenvironment{eq}{\begin{equation}}{\end{equation}}
\renewcommand{\Ref}[1]{(\ref{#1})}

\newcommand{\Char}{\mathop{\rm char}}

\newcommand{\ver}{\mathop{\rm ver}}
\newcommand{\arr}{\mathop{\rm arr}}

\newcommand{\FF}{\mathbb{F}}

\newcommand{\CC}{\mathbb{C}}
\newcommand{\ZZ}{\mathbb{Z}}

\newcommand{\calcA}{\mathcal{A}}
\newcommand{\calcB}{\mathcal{B}}
\newcommand{\calcL}{\mathcal{L}}

\newcommand{\Set}{\mathrm{Set}}    
\newcommand{\GL}{{\rm GL}}         
\newcommand{\SL}{{\rm SL}}         
\newcommand{\SO}{{\rm SO}}         
\newcommand{\Sp}{{\rm Sp}}         
\newcommand{\can}{{\rm can}}               

\newcommand{\calcP}[1]{\mathcal{P}_{#1}}  

\newcommand{\F}[1]{\mathcal{F}_{#1}}  
\newcommand{\Alt}{\mathrm{Alt}}       

\newcommand{\si}{\sigma}
\newcommand{\al}{\alpha}
\newcommand{\be}{\beta}
\newcommand{\ga}{\gamma}
\newcommand{\la}{\lambda}
\newcommand{\de}{\delta}

\newcommand{\ov}[1]{\overline{#1}}
\newcommand{\un}[1]{{\underline{#1}} }

\newcommand{\algX}[1]{{\rm alg}_{\FF}\{X\}_{#1}}

\newcommand{\diag}{\mathop{\rm diag}}
\newcommand{\tr}{\mathop{\rm tr}}
\newcommand{\rank}{\mathop{\rm rank}}

\newcommand{\matr}[4]{\left(
\begin{array}{cc}
#1 & #2 \\ 
#3 & #4 \\ 
\end{array}
\right)}

\newcommand{\matrTri}[9]{
\left(\begin{array}{ccc}
#1 &  #2 & #3 \\
#4 &  #5 & #6 \\
#7 &  #8 & #9 \\
\end{array}
 \right)
} 

\newcommand{\TD}{{\rm TD}} 

\begin{document}
\title[Pairs of matrices with simple spectrum]{Pairs of matrices with simple spectrum}

\thanks{The second author was supported by FAPESP 2021/01690-7, FAEPEX 2429/22}


\author{Jennyfer Juliana Calderón Moreno}
\address{Jennyfer Juliana Calderón Moreno\\ State University of Campinas, 651 Sergio Buarque de Holanda, 13083-859 Campinas, SP, Brazil}
\email{j235404@dac.unicamp.br (Jennyfer Juliana Calderón Moreno)}

\author{Artem Lopatin}
\address{Artem Lopatin\\ State University of Campinas, 651 Sergio Buarque de Holanda, 13083-859 Campinas, SP, Brazil}
\email{dr.artem.lopatin@gmail.com (Artem Lopatin)}

\begin{abstract} We established two-sided Curto--Herrero conjecture for pairs of matrices, where the first matrix has a simple spectrum. Namely, it is shown that these pairs are separated by ranks of non-commutative polynomials in matrices.  
Moreover, we provided some upper bound on degrees of non-commutative polynomials which should be considered.


\end{abstract}

\maketitle

\section{Introduction}

All vector spaces and matrices are over an arbitrary (possibly finite) field $\FF$ of characteristic $\Char{\FF}\geq0$, unless otherwise stated. We fix some  total order over $\FF$ considered as an abstract set.  By an algebra we always mean an associative algebra.

\subsection{Similarity of $m$-tuples of matrices}\label{section_conj}

The space $M_n^m$ of $m$-tuples of $n\times n$ matrices with $n\geq2$ is a $\GL_n$-module with respect to the diagonal conjugation: $g\cdot (A_1,\ldots,A_m)=(gA_1 g^{-1},\ldots, gA_mg^{-1})$. It is well-known that the problem of classification of all $\GL_n$-orbits on $M_n^m$ for $m\geq2$ is ``impossible''~\cite{LeBruyn_1995}, since this classification problem for matrix pairs contains the classification problem for arbitrary systems of linear mappings on vector spaces~\cite{Belitskii_Sergeichuk_2003, Gelfand_Ponomarev_1969}.  

In 1985 Curto and Herrero  formulated a conjecture (see Conjecture 8.14 of~\cite{CurtoHerrero_1985}) that $\un{A}\in M_n^m$ lies in the closure of $\GL_n$-orbit of $\un{B}\in M_n^m$ if and only if $\rank f(\un{A}) \leq \rank f(\un{B})$ for every (noncommutative) polynomial $f$ in $m$ variables. Hadwin and Larson in 2003 gave a counterexample (see Example 5 of~\cite{HadwinLarson_2003}) to the weaker two-sided Curto--Herrero conjecture:

\begin{conj}[Two-sided Curto--Herrero conjecture]\label{conj_CH}
For every $\un{A},\un{B}\in M_n^m$ we have that $\GL_n$-orbits of  $\un{A}$ and $\un{B}$ coincide if and only if $\rank f(\un{A}) = \rank f(\un{B})$ for every (noncommutative) polynomial $f$ in $m$ variables (possibly with a free term). 
\end{conj}%

\noindent{}Namely, Conjecture~\ref{conj_CH} does not hold in general for $n\geq3$ and $m\geq 2$. For the sake of completeness we present a modified version of the counterexample of Hadwin and Larson in the proof of Lemma~\ref{lemma_ex1}. On the other hand, Conjecture~\ref{conj_CH} still may be valid for some particular classes of $m$-tuples of matrices.   Note that using the rational canonical form of a matrix it is easy to see that Conjecture~\ref{conj_CH} holds for $m=1$.

Let us remark that recently Derksen, Klep, Makam, Vol\v{c}i\v{c} proved the two-sided Hadwin-Larson conjecture that claims that $\GL_n$-orbits of  $\un{A},\un{B}\in M_n^m$ coincide if and only if $\rank F(\un{A}) = \rank F(\un{B})$ for every matrix $F$ of non-commutative polynomials in $m$ variables (see Theorem 1.1 of~\cite{DerksenKlepMakamVolcic_2023}).

\subsection{Pair of matrices}\label{section_pairs}

 In~\cite{Belitskii_1983, Belitskii_2000} Belitskii established the algorithm which associates to every pair $\un{A}$ of $M_n^2$ the unique canonical pair $\un{A}_{\can}$ such that
\begin{enumerate}
\item[$\bullet$] $\un{A}_{\can}$ lies in the $\GL_n$-orbit of $\un{A}$;

\item[$\bullet$] a pair $\un{B}\in M_n^2$ belongs to the $\GL_n$-orbit of $\un{A}$ if and only if $\un{A}_{\can}=\un{B}_{\can}$.
\end{enumerate} 
But as it was mentioned in~\cite{Sergeichuk_2011}, without restrictions on the size of matrices, we cannot expect to have explicit descriptions of Belitskii’s canonical matrix pairs. Note that in~\cite{Sergeichuk_2000} Belitskii’s algorithm was expanded to the classification problem of arbitrary systems of linear mappings.

Denote by $D_n$ the subset of $M_n^2$ consisting of all pairs $\un{A}=(A_1,A_2)$ such that the spectrum of $A_1$ is simple, i.e.,  $A_1$ has $n$ distinct eigenvalues in $\FF$. In particular, we assume that the following convention holds:

\begin{conv}\label{conv} There are at least $n$ pairwise different elements in $\FF$. 
\end{conv}

\noindent{}Obviously, $D_n$ is a $\GL_n$-submodule of $M_n^2$. Futorny, Horn, Sergeichuk in Theorem 6.1 of~\cite{Sergeichuk_2011} described all canonical pairs for $D_n$. Canonical pairs are parameterized by 
\begin{enumerate}
\item[$\bullet$] directed graphs $G$ without undirected cycles, which we call the {\it type} of the canonical pair;

\item[$\bullet$] finite number of constants of $\FF$, which can take arbitrary values of $\FF$; we refer to these constants as the {\it free parameters} for the type $G$.
\end{enumerate}
More details can be found in Section~\ref{section_space}. Note that $\un{A}\in D_n$ is indecomposable if and only if the type of $\un{A}$ is a tree (see Section 6 of~\cite{Sergeichuk_2011}).

\subsection{Polynomial matrix invariants}\label{section_matr_inv}

The coordinate algebra of the vector space $M_n^m$ is the polynomial $\FF$-algebra $\calcP{n,m}=\FF[M_n^m]=\FF[x_{ij}(k)\,|\,1\leq i,j\leq n,\; 1\leq k\leq m]$. To define the action of $\GL_n$ we consider $\calcP{n,m}$ as the symmetric algebra $S((M_n^m)^{\ast})$ over the dual space $(M_n^m)^{\ast}$, with a basis $\{x_{ij}(k)\}$ of $(M_n^m)^\ast$, where $x_{ij}(k):M_n^m\to\FF$ is defined by $\un{A}\to (A_k)_{ij}$.  Here $A_{ij}$ stands for the $(i,j)^{\rm th}$ entry of a matrix $A$.  Thus every element of $\calcP{n,m}$ defines a polynomial functions $M_n^m \to \FF$. The algebra $\calcP{n,m}$ becomes a $\GL_n$-module with
$(g \cdot f)(\un{A}) = f(g^{-1}\cdot \un{A})$ for all $g\in\GL_n$, $f \in (M_n^m)^\ast$ and $\un{A} \in M_n^m$.
The algebra of {\it polynomial matrix invariants} is
$$\calcP{n,m}^{\GL_n}=\{f\in \calcP{n,m} \mid g\cdot f=f \text{ for all }g\in \GL_n\}.$$ 
Note that $\calcP{n,m}^{\GL_n}$ lies in the algebra of all $f\in \calcP{n,m}$ with $f(g\cdot\un{A})=f(\un{A})$ for all $g\in\GL_n$, $\un{A}\in M_n^m$, where we have equality of these two algebras in the case of an infinite field $\FF$. 

To define some polynomial matrix invariants consider the {\it generic $n\times n$ matrices}
$$X_k= (x_{ij}(k))_{1\leq i,j\leq n}$$%
for $1\leq k\leq m$. Denote by $\algX{m}$ the associative $\FF$-algebra generated by the generic matrices $X_1,\ldots,X_m$ together with the identity matrix $I_n$.  Any product of the generic matrices is called a word of $\algX{m}$. The unit $I_n\in \algX{m}$ is called the empty word. Since $x_{ij}(k):M_n^m\to\FF$ is a map, for every $F\in \algX{m}$ and $\un{A}\in M_n^m$ we define $F(\un{A})\in M_n$.  We define the  {\it degree} of $F\in\algX{m}$ as the minimal degree of a non-commutative polynomial $f$ in letters $x_1,\ldots,x_m$ such that $F=f(X_1,\ldots,X_m)$. Equivalently, the  degree of $F\in\algX{m}$ is the maximum of degrees of $F_{ij}$ for $1\leq i,j\leq n$. Given $1\leq t\leq n$, we write $\si_t$ for the $t^{\rm th}$ coefficient of the characteristic polynomial of an $n\times n$ matrix, i.e., $\si_1=\tr$ is the trace and $\si_n=\det$ is the determinant. 


We define the $\GL_n$-action on $M_n(\calcP{n,m})$ as follows: for $g\in \GL_n$ and $F\in M_n(\calcP{n,m})$ we set that $g\cdot F$ is the matrix with $(i,j)^{\rm th}$ entry equal to $g\cdot F_{ij}$. Then the action of $\GL_n$ on $\calcP{n,m}$ is explicitly described by the next formula, which can be easily obtained:
$$g\cdot X_k=g^{-1} X_k g$$
for all $g\in\GL_n$ and $1\leq k\leq n$. Note that for every $F,H$ from $\algX{m} \subset M_n(\calcP{n,m})$, $g\in\GL_n$ and $1\leq t\leq n$ we have
\begin{eq}\label{eq_algZ}
\si_t(g\cdot F)=\si_t(F)\;\text{ and }\;  g\cdot (FH)=(g\cdot F) (g\cdot H).
\end{eq}%

The generators for the algebra $\calcP{n,m}^{\GL_n}$ of polynomial matrix invariants were described by Sibirskii~\cite{Sibirskii68} and Procesi~\cite{Procesi76} in case $\Char\FF=0$ and by Donkin~\cite{Donkin92a} in case $\Char\FF>0$. Namely, the algebra $\calcP{n,m}^{\GL_n}$ is generated by $\si_t(U)$, where $U$ is a product of generic matrices and $1\leq t\leq n$. A minimal generating sets for $\calcP{2,m}^{\GL_n}$ and $\calcP{3,m}^{\GL_n}$ were obtained in~\cite{Procesi84, DKZ02, Lopatin_Sib, Lopatin_Comm1, Lopatin_Comm2} and for $\calcP{4,2}^{\GL_n}$, $\calcP{5,2}^{\GL_n}$ were obtained in~\cite{Teranishi86, Drensky_Sadikova_4x4, Djokovic07, Djokovic09} in case $\Char\FF=0$. Note that the algebra $\calcP{n,m}^{\GL_n}$ can be generalized to the algebra of invariants of representations of quivers with respect to the action of the classical linear groups $\GL_n$, $\SL_n$, ${\rm O}_n$, $\SO_n$, $\Sp_n$ (see~\cite{Lopatin_2009JA} the references and more details).

\subsection{Abstract matrix invariants}\label{section_abstr_inv}

Denote by $\F{n,m}$ the $\FF$-algebra of all functions $M_n^m\to \FF$ with respect to the pointwise addition and multiplication. The algebra of {\it abstract matrix invariant functions} (for short, {\it abstract matrix invariants}) is
$$\F{n,m}^{\GL_n}=\{f\in \F{n,m}\,|\,f(g\cdot\un{A})=f(\un{A}) \text{ for all }g\in\GL_n,\; \un{A}\in M_n^m\}.$$
Obviously, $\calcP{n,m}^{\GL_n}$ lies in $\F{n,m}^{\GL_n}$.  The composition of generic matrices and an abstract invariant is again an abstract invariants. Namely,

\begin{remark}\label{remark_abs_inv2}
If $h\in \F{n,m}^{\GL_n}$ an abstract invariant and $F_1,\ldots,F_d\in\algX{m}$, then the map $f=h(F_1,\ldots,F_d)\in \F{n,m}$ given by $f(\un{A})=h(F_1(\un{A}),\ldots,F_d(\un{A}))$, where $\un{A}\in M_n^m$, is an abstract invariant from $\F{n,m}^{\GL_n}$.  
\end{remark}

Let us construct some abstract invariants. Define the  function $\zeta$ of $\F{n,1}$ as follows:
$$
\zeta(A)= \left\{
\begin{array}{rl}
1, & \text{if } A=0 \\
0, & \text{if } A\neq0 \\
\end{array}
\right..
$$

\begin{remark}\label{remark_rank}
The rank of a matrix is a map $\rank:M_n\to \ZZ$, but in general $\rank\not\in\F{n,1}$. Nevertheless, for every $0\leq t\leq n$ we define $\rank_t$ of $\F{n,1}$  by
$$
\rank\nolimits_t(A)= \left\{
\begin{array}{rl}
1, & \text{if } \rank(A)=t \\
0, & \text{otherwise } \\
\end{array}
\right..
$$
Obviously, $\rank(A)=\rank(B)$  if and only if $\rank_t(A)=\rank_t(B)$ for all $0\leq t\leq n$. Note that $\rank_0(X_1)=\zeta(X_1)$.
\end{remark}

\begin{remark}\label{remark_abs_inv}
For every $F\in\algX{m}$ and $0\leq t\leq n$ we have that  $\zeta(F)$ and $\rank_t(F)$ belong to  $\F{n,m}^{\GL_n}$.  If $\FF$ is infinite, then  $\zeta(X_1)$ does not lie in $\calcP{n,m}^{\GL_n}$, since any non-zero polynomial in one variable do not have infinitely many roots.
\end{remark}

\begin{defin} Given an abstract invariant $f\in \F{n,m}^{\GL_n}$, the {\it width} of $f$ is the smallest $d>0$ such that there exists an abstract invariant $h\in \F{n,d}^{\GL_n}$ and $F_1,\ldots,F_d\in\algX{m}$ such that 
$$f = h(F_1,\ldots,F_d).$$
\end{defin}

Note that $\si_t(F)$ for $1\leq t\leq n$ as well as $\rank_t(F)$ for $0\leq t\leq n$  have width one and the width of any $f\in \F{n,m}^{\GL_n}$ is less or equal to $m$. 

We say that a set $S\subset \F{n,m}^{\GL_n}$ of abstract invariants {\it separates $\GL_n$-orbits} on $M_n^m$ if for every $\un{A},\un{B}\in M_n^m$ the condition $f(\un{A})=f(\un{B})$ for all $f\in S$ implies that $\GL_n\cdot\un{A}=\GL_n\cdot\un{B}$. Given a $\GL_n$-invariant subset $L$ of  $M_n^m$, we similarly define when a set $S\subset\F{n,m}^{\GL_n}$ or a set $S\subset \{\rank(F)\,|\,F\in\algX{m}\}$ separates $\GL_n$-orbits on $L$.   Separating polynomial invariants were introduced by Derksen and Kemper in~\cite{DerksenKemper_book} (see~\cite{DerksenKemper_bookII} for the second edition) and then were studied in~\cite{Cavalcante_Lopatin_1, DM5, Domokos_2017, Domokos20, Domokos20Add,  Dufresne_Elmer_Sezer_14,   kaygorodov2018, Kemper_Lopatin_Reimers_2022, Kohls_Sezer_2013, Lopatin_Reimers_1}.  Note that separating polynomial invariants do not separate orbits with non-empty intersection of closures in Zarisky topology.

\begin{example}\label{ex1}
Assume that $\un{A},\un{B}\in D_n$ and eigenvalues of $A_1$ are not the same as eigenvalues of $B_1$. Then $\un{A}$, $\un{B}$ are separated by $\si_t(X_1)$, where $1\leq t\leq n$. 
\end{example}


\subsection{Results}\label{section_results}

In Theorem~\ref{theo_main2} we proved the two-sided Curto--Herrero conjecture for the space $D_n$ of pairs of matrices, where the first matrix has simple spectrum.  As a consequence, in Corollary~\ref{cor_theo_main2} we showed that $\GL_n$-orbits on $D_n$ are separated by abstract invariants of width one.  Moreover, we provided some upper bound on degrees of non-commutative polynomials which should be considered. We also proved that the types of elements of $D_n$ are separated by the smaller set $\Set_{\si}\cup \Set_{\zeta}$ (see Theorem~\ref{theo_main1} for the definitions). On the other hand, in Lemma~\ref{lemma_ex1} we proved that $\GL_n$-orbits on $M_n^m$ are not separated by abstract invariants of width one in case $n\geq 3$ and $m\geq2$. Lemma~\ref{lemma_ex} implies that $\Set_{\si}\cup \Set_{\zeta}$ does not separate $\GL_n$-orbits on $D_n$ for $n\geq4$.

\section{The space of canonical pairs}\label{section_space}


A $\ast$-matrix is an $n\times n$ matrix $\calcA$ such that  $(i,j)^{th}$ entry $\calcA_{ij}$ of $\calcA$ is one of the following symbols: $0,1,\ast$ ($1\leq i,j\leq n$). Equivalently, a $\ast$-matrix $\calcA$ can be considered as the set of all matrices $A\in M_n$ with 
$$A_{ij}=\left\{
\begin{array}{cc}
0,& \text{ if } \calcA_{ij}=0 \\
1,& \text{ if } \calcA_{ij}=1 \\
\end{array}
\right.$$
for all $1\leq i,j\leq n$. In case $\calcA_{ij}=\ast$ we say that $A_{ij}\in\FF$ is a {\it free parameter} for $\calcA$.

By graph we always mean directed finite graph with numbered vertices $v_1,\ldots,v_n$. As an example, one-arrow graphs $v_1\to v_2$ and $v_2\to v_1$ are different. The set of vertices of a graph $G$ is denoted by $\ver(G)$ and the set of arrows is denoted by $\arr(G)$. A graph without undirected cycles is called a {\it forest}. If any two vertices of a forest $G$ are connected by an undirected path, then $G$ is called a tree. Introduce the following lexicographical order on the set of all pairs $P_n=\{(i,j) \,|\, 1\leq i,j\leq n\}$: $(i,j)<(r,s)$ if $i<r$ or $i=r$, $j<s$.

A $\ast$-matrix $\calcA$ is {\it canonical} if there exists a forest $G=G_{\!\calcA}$ with vertices $v_1,\ldots,v_n$ such that for every $1\leq i\neq j\leq n$ we have 
\begin{enumerate}
\item[(i)] $\calcA_{ij}=1$ if $G$ has the arrow $v_i\to v_j$; 

\item[(ii)] $\calcA_{ij}=0$ if one of the following conditions holds: 
\begin{enumerate}
\item[(a)] $G$ has no undirected path from $v_i$ to $v_j$;

\item[(b)] the exists (a unique) undirected path $a$ from $v_i$ to $v_j$ and $a$ contains an arrow $v_r\to v_s$ such that $(r,s)> (i,j)$;
\end{enumerate}

\item[(iii)] $\calcA_{ij}=\ast$, if there exists (a unique) undirected path $a$ from $v_i$ to $v_j$ and for every arrow $v_r\to v_s$ of $a$ we have $(r,s) < (i,j)$;

\item[(iv)] $\calcA_{ii}=\ast$.
\end{enumerate}%
\noindent{}Note that $\calcA \to G_{\!\calcA}$ is the 1-1 correspondence between  $n\times n$ canonical  $\ast$-matrices and forests with $n$ vertices.

The following lemma is a consequence of Theorem 6.1 of~\cite{Sergeichuk_2011}. For the sake of completeness we present a straightforward proof of the lemma.

\begin{lemma}\label{lemma1}
If $\calcA$ and $\calcB$ are $n\times n$ canonical  $\ast$-matrices, then $\calcA \cap \calcB = \varnothing$ in case $\calcA\neq\calcB$.
\end{lemma}
\begin{proof} Assume that $A\in \calcA \cap \calcB$ and  $\calcA\neq\calcB$. Vertices of $G_{\!\calcA}$ and $G_{\calcB}$ we denote by $v_1,\ldots,v_n$. Modulo permutation of $\calcA$ and $\calcB$, without loss of generality we assume can that there exists the pair $(i,j)\in P_n$ such that 
\begin{enumerate}
\item[$\bullet$] $v_i\to v_j$ is an arrow of $G_{\!\calcA}$ and is not an arrow of $G_{\calcB}$;

\item[$\bullet$] for every $(r,s)\in P_n$ with $(r,s)<(i,j)$ we have that $v_r\to v_s$ is an arrow of $G_{\!\calcA}$ if and only if  $v_r\to v_s$ is an arrow of  $G_{\calcB}$.
\end{enumerate}

Since $\calcA_{ij}=1$, we have $A_{ij}=1$ and $\calcB_{ij}=\ast$. Then there exists a unique undirected path $b$ in $G_{\calcB}$ from $v_i$ to $v_j$ such that for every arrow $v_r\to v_s$ of $b$ we have $(r,s)<(i,j)$; in particular, $v_r\to v_s$ lies in $G_{\!\calcA}$. Therefore, $b$ is also an undirected path in $G_{\!\calcA}$. Since $v_i\to v_j$ is not an arrow of $b$, we obtain that $G_{\!\calcA}$ is not a forest; a contradiction.
\end{proof}

A pair $\un{A}$ of $D_n$ is called {\it canonical}, if $A_1=\diag(a_1,\ldots,a_n)$ is a diagonal matrix with $a_1< \cdots <a_n$ and $A_2\in \calcA$ for some canonical $\ast$-matrix $\calcA$. In this case the graph $G_{\!\calcA}$ is called the {\it type} of $\un{A}$. The following statement is Theorem 6.1 of~\cite{Sergeichuk_2011}, which was proven over $\CC$ but the proof remains valid over an arbitrary field $\FF$.

\begin{theo}\label{theoFutorny}
The set of all canonical pairs of $D_n$ is in 1-1 correspondence with $\GL_n$-orbits on $D_n$. 
\end{theo}

\begin{example}\label{ex_theoFutorny}
Assume $n=2$. Then the set of all canonical $\ast$-matrices $\calcA$ and the corresponding graphs $G_{\!\calcA}$ is given below:
\begin{enumerate}
\item[$\bullet$] $\calcA=\matr{\ast}{0}{0}{\ast}\;$ with  $\;G_{\!\calcA}:v_1\qquad v_2$, i.e., there is no arrows in $G_{\!\calcA}$;

\item[$\bullet$] $\calcA=\matr{\ast}{1}{\ast}{\ast}\;$ with  $\;G_{\!\calcA}:v_1\longrightarrow v_2$;

\item[$\bullet$] $\calcA=\matr{\ast}{0}{1}{\ast}\;$ with  $\;G_{\!\calcA}:v_1 \longleftarrow v_2$.
\end{enumerate}
\end{example}

\begin{example}\label{ex2_theoFutorny}
For $n=4$ we present some canonical $\ast$-matrices $\calcA$ and the corresponding graphs $G_{\!\calcA}$:
\begin{enumerate}
\item[$\bullet$] $\calcA = \left(
\begin{array}{cccc}
\ast & 1 & 0 & 0 \\
\ast & \ast & 1 & 0 \\
\ast & \ast & \ast & 1 \\
\ast & \ast & \ast & \ast \\
\end{array}
\right)\;$ with $\;G_{\!\calcA}: v_1 \longrightarrow v_2 \longrightarrow v_3 \longrightarrow v_4$; 

\item[$\bullet$] $\calcA = \left(
\begin{array}{cccc}
\ast & 0 & 0 & 0 \\
   1 & \ast & 0 & 0 \\
0 & 1 & \ast & 0 \\
0 & 0 & 1 & \ast \\
\end{array}
\right)\;$ with $\;G_{\!\calcA}: v_1 \longleftarrow v_2 \longleftarrow v_3 \longleftarrow v_4$. 
\end{enumerate}
\end{example}

%
%

\section{Combinatorics of matrices}\label{section_combinatorics}

Denote by $E_{ij}$ the $n\times n$ matrix such that its the only non-zero entry is $1$ in the position $(i,j)$. 
A word $w$ in letters $x_1,\ldots,x_k$ is called multilinear if letter $x_i$ appears in $w$ at most one time for all $1\leq i\leq k$.  Given a sequence $S=(A_1,\ldots,A_k)$ of matrices and a word $w$ in letters $x_1,\ldots,x_k$, we denote by 
\begin{enumerate}
\item[$\bullet$] $w(S)$ the result of substitutions $x_1\to A_1,\ldots,x_k\to A_k$ in $w$; 

\item[$\bullet$] $\Alt(S)=A_1-A_2+A_3-A_4+ \ldots + (-1)^{k+1}A_k$ the alternating sum. 
\end{enumerate}

Given a matrix $A$ and a symbol $\de\in\{1,\mathsf{T}\}$, we denote
$$A^{\de}=\left\{
\begin{array}{ccl}
A &, & \text{if }\de=1 \\
A^{\mathsf{T}}&, &\text{if } \de=\mathsf{T} \\
\end{array}
\right..
$$

Given  $\un{i}=(i_1,\ldots,i_{k+1})$ with pairwise different elements $1\leq i_1,\ldots,i_{k+1}\leq n$, $k\geq 1$, and $\un{\de}\in\{1,\mathsf{T}\}^k$, the sequence of matrices $\TD_{\un{i},\un{\de}}=(E_{i_1i_2}^{\de_1},E_{i_2i_3}^{\de_2},\ldots,E_{i_{k}i_{k+1}}^{\de_k})$ is called a {\it three-diagonal sequence} of length $k$. Note that in case $\un{i}=(1,\ldots,k+1)$ all elements of $\TD_{\un{i},\un{\de}}$ are three-diagonal matrices. 

A three-diagonal sequence $\TD_{\un{i},\un{\de}}$ is called a {\it staircase} sequence with {\it foundation} $E_{i_1i_{k+1}}$ if $k\geq3$ is odd, $\de_1=\de_3=\cdots=\de_k=1$  and $\de_2=\de_4=\cdots=\de_{k-1}=\mathsf{T}$. 

\begin{example} {\bf (a)} $S=(E_{12}, E_{23}^{\mathsf{T}}, E_{34})$ is a staircase sequence with foundation $C=E_{14}$ for $k=3$ and $\un{i}=(1,2,3,4)$. In case $n=4$ and $\al\in\FF$ we have that 
$$\Alt(S) + \al C =
\left(\begin{array}{cccc}
0 & 1 & 0 & \Circled{\al} \\
0 & 0 & 0 & 0 \\
0 & -1 & 0 & 1 \\
0 & 0 & 0 & 0 \\
\end{array}
\right),$$
where the coefficient of the foundation is encircled. 

\medskip
\noindent{{\bf (b)} } $S=(E_{12}, E_{23}^{\mathsf{T}}, E_{34}, E_{45}^{\mathsf{T}}, E_{56})$ is a staircase sequence with foundation $C=E_{16}$ for $k=5$ and $\un{i}=(1,2,3,4,5,6)$. In case $n=6$ and $\al\in\FF$ we have that 
$$\Alt(S) +\al C  = 
\left(\begin{array}{cccccc}
0 & 1 & 0 & 0 & 0 & \Circled{\al}\\
0 & 0 & 0 & 0  & 0 & 0\\
0 & -1 & 0 & 1 & 0 & 0\\
0 &   0   & 0 & 0     & 0 & 0\\
0 &   0   & 0 & -1 & 0 & 1\\
0 &   0   & 0 & 0     & 0 & 0\\
\end{array}
\right),$$
where the coefficient of the foundation is encircled. 
\end{example}

\begin{lemma}\label{lemma_TD_T}
Assume that $S=\TD_{\un{i},\un{\de}}$ is a three-diagonal sequence of length $k\geq1$ with $\de_1=\cdots=\de_k=\mathsf{T}$. Then there are words $u_1,u_2$ in letters $x_1,\ldots,x_k$ such that 
$$u_1(S) E_{i_1i_{k+1}} u_2(S) = E_{i_ri_{r+1}}^{\de_r}$$ 
for $r=1$; moreover, 
$$
\deg(u_1) + \deg(u_2)= k+1.
$$
\end{lemma}
\begin{proof}
For $u_1=x_1$ and $u_2=x_k x_{k-1}\cdots x_{1}$ we have that
$$
u_1(S) \,E_{i_1i_{k+1}} \, u_2(S) = E_{i_1i_2}^{\mathsf{T}}\cdot E_{i_1i_{k+1}}\cdot E_{i_{k}i_{k+1}}^{\mathsf{T} }E_{i_{k-1}i_{k}}^{\mathsf{T}} E_{i_{k}i_{k+1}}^{\mathsf{T}}\cdots E_{i_{1}i_{2}}^{\mathsf{T}} = E_{i_1i_{2}}^{\mathsf{T}}.
$$
\noindent{}The required is proven.
\end{proof}

\begin{lemma}\label{lemma_TD}
Assume that $S=\TD_{\un{i},\un{\de}}$ is a three-diagonal sequence of length $k\geq1$ with $\de_l=1$ for some $1\leq l\leq k$. Then we have one of the following two cases:
\begin{enumerate}
\item[(a)] there are (possibly empty) words $u_1,u_2$ in letters $x_1,\ldots,x_k$ such that 
$$u_1(S) E_{i_1i_{k+1}} u_2(S) = E_{i_ri_{r+1}}^{\de_r}$$ 
for some $1\leq r\leq k$; moreover, the word $u_1u_2$ is multilinear; 

\item[(b)]there are  non-empty words $w_1,\ldots,w_r$, where $r\geq1$, in letters $x_1,\ldots,x_k$ and (possibly empty) words $u_1,u_2$ in letters $x_1,\ldots,x_k$ such that 
$$(w_1(S),\ldots,w_r(S))$$ 
is a staircase sequence with foundation $u_1(S) E_{i_1i_{k+1}} u_2(S)$;  moreover, the word $u_1u_2w_1\cdots w_r$ is multilinear; 
\end{enumerate}
\end{lemma}
\begin{proof}  We prove the lemma by the induction on length $k$ of the sequence $S$. For short, denote $C=E_{i_1i_{k+1}}$.

Let $k=1$. Then $\de_1=1$ and case (a) holds for $r=1$ and $u_1=u_2=1$. 

Therefore, we can assume that $k\geq2$. The following Claim 1 and Claim 2 imply that to complete the proof it is enough to consider the case of odd $k$ with  $\de_1=\de_3 = \cdots = \de_k=1$ and $\de_2=\de_4=\cdots=\de_{k-1}=\mathsf{T}$. Obviously, these conditions imply that case (b) holds for $w_1=x_1,\ldots,w_k=x_k$ and $u_1=u_2=1$, where $r=k$. Since inequality~\Ref{eq_lemma_TD_b} also is valid, the proof is completed.  

\medskip
\noindent{\bf Claim 1.} {\it Without loss of generality, we can assume that $\de_1=\de_k=1$}.
\medskip

\noindent{}To prove Claim 1 we consider the following three cases.
\begin{enumerate}
\item[{\bf (i)}] Let $\de_1=1$ and $\de_k=\mathsf{T}$. Denote 
$$l=\max\{ 1\leq s<k \,|\, \de_s=1\}.$$
Since $\de_1=1$,  we have that $l$ is well-defined. Then for $u_1=1$ and $u_2=x_k x_{k-1}\cdots x_{l+1}$, we have that
\begin{eq}\label{lemma_TDS_for_claim1_eq1}
u_1(S) \,C\, u_2(S) = E_{i_1i_{k+1}}\cdot E_{i_{k}i_{k+1}}^{\mathsf{T}} 
E_{i_{k-1}i_{k}}^{\mathsf{T}} \cdots E_{i_{l+1}i_{l+2}}^{\mathsf{T}} =E_{i_1i_{l+1}}
\end{eq}%
and $\de_1=\de_{l}=1$. Consider the three-diagonal sequence 
$$S'=\TD_{(i_1,\ldots,i_{l+1}),(\de_1,\ldots,\de_l)},$$ 
which has length $1\leq l<k$. By induction, case (a) or case (b) is valid for $S'$. In both cases we can obtain that the same case also holds for $S$. Namely, 
\begin{enumerate}
\item[$\bullet$] if case (a) holds for $S'$, i.e., for some (possibly empty) words $u'_1$, $u'_2$ in letters $x_1,\ldots,x_l$ and some  $1\leq r\leq l$ we have that
$u'_1(S) E_{i_1,i_{l+1}} u'_2(S) = E_{i_r, i_{r+1}}^{\de_r}$, since $u'_1(S)= u'_1(S')$ and $u'_2(S)=u'_2(S')$; thus  case (a) holds for $S$ for $u_1''=u'_1 u_1$ and $u_2''=u_2 u'_2$ by formula~\Ref{lemma_TDS_for_claim1_eq1}, since $u_1''u_2''$ is a miltilinear word;

\item[$\bullet$] if case (b) holds for $S'$, i.e., there are words $w_1,\ldots,w_r$ in letters $x_1,\ldots,x_l$ and (possibly empty) words $u'_1$, $u'_2$ in letters $x_1,\ldots,x_l$ such that $(w_1(S),\ldots,w_r(S))$ is a staircase sequence with the foundation  
$u'_1(S) E_{i_1,i_{l+1}} u'_2(S)$; thus  case (b) holds for $S$ for $w_1,\ldots,w_r$ and $u_1''=u'_1 u_1$, $u_2''=u_2 u'_2$ by formula~\Ref{lemma_TDS_for_claim1_eq1}, since $u_1''u_2''w_1\cdots w_r$ is a miltilinear word.
\end{enumerate}

\item[{\bf (ii)}] $\de_1=\mathsf{T}$ and $\de_k=1$.  Denote 
$$l=\min\{ 1< s\leq k \,|\, \de_s=1\}.$$
Since $\de_k=1$,  we have that $l$ is well-defined. Then for $u_1=x_{l-1} \cdots x_2x_1$ and $u_2=1$, we have that
\begin{eq}\label{lemma_TDS_for_claim1_eq2}
u_1(S) \,C\, u_2(S) =  E_{i_{l-1}i_{l}}^{\mathsf{T}}\cdots E_{i_2i_3}^{\mathsf{T}} E_{i_1i_2}^{\mathsf{T}}\cdot E_{i_1i_{k+1}} = E_{i_{l}i_{k+1}}
\end{eq}%
and $\de_l=\de_{k}=1$. Consider the three-diagonal sequence 
$$S'=\TD_{(i_l,\ldots,i_{k+1}),(\de_l,\ldots,\de_k)},$$ 
which has length $k'=k-l+1<k$, where $k'\geq1$. By induction, case (a) or case (b) is valid for $S'$. In both cases we obtain that the same case also holds for $S$. Namely, 
\begin{enumerate}
\item[$\bullet$] if case (a) holds for $S'$, i.e., for some (possibly empty) words $u'_1$, $u'_2$ in letters $x_{l},\ldots,x_{k}$ and some $l\leq r\leq k$ we have that 
$u'_1(S) E_{i_l,i_{k+1}} u'_2(S) = E_{i_r, i_{r+1}}^{\de_r}$; thus case (a) holds for $S$ for $u_1''=u'_1 u_1$ and $u_2''=u_2 u'_2$ by formula~\Ref{lemma_TDS_for_claim1_eq2}, since $u_1''u_2''$ is a miltilinear word;  

\item[$\bullet$] if case (b) holds for $S'$, i.e., there are words $w_1,\ldots,w_r$ in letters $x_l,\ldots,x_k$ and (possibly empty) words $u'_1$, $u'_2$ in letters $x_l,\ldots,x_k$ such that $(w_1(S),\ldots,w_r(S))$ is a staircase sequence with the foundation  
$u'_1(S) E_{i_l,i_{k+1}} u'_2(S)$; thus case (b) holds for $S$ for $w_1,\ldots,w_r$ and $u_1''=u'_1 u_1$, $u_2''=u_2 u'_2$ by formula~\Ref{lemma_TDS_for_claim1_eq2}, since $u_1''u_2''w_1\cdots w_r$ is a miltilinear word.
\end{enumerate}

\item[{\bf (iii)}] $\de_1=\mathsf{T}$ and $\de_k=\mathsf{T}$.  Denote
$$l=\min\{ 1< s<k \,|\, \de_s=1\} \text{ and } t=\max\{ 1< s<k \,|\, \de_s=1\}.$$%
\noindent{}Then for $u_1=x_{l-1} \cdots x_2x_1$ and $u_2=x_k x_{k-1}\cdots x_{t+1}$ we can see that $u_1(S) \,C\, u_2(S)$ is equal to 
\begin{eq}\label{lemma_TDS_for_claim1_eq3}
E_{i_{l-1}i_{l}}^{\mathsf{T}}\cdots E_{i_1 i_2}^{\mathsf{T}} E_{i_2 i_3}^{\mathsf{T}} \cdot E_{i_1i_{k+1}}
\cdot  E_{i_{k}i_{k+1}}^{\mathsf{T}} E_{i_{k+1}i_{k+2}}^{\mathsf{T}}\cdots
E_{i_{t+1}i_{t+2}}^{\mathsf{T}} =E_{i_{l}i_{t+1}}
\end{eq}%
and $\de_l=\de_{t}=1$. Consider the three-diagonal sequence 
$$S'=\TD_{(i_l,\ldots,i_{t+1}),(\de_l,\ldots,\de_t)},$$ 
which has length $k'=t-l+1<k$. By induction, case (a) or case (b) is valid for $S'$. In both cases we obtain that the same case also holds for $S$. Namely, 
\begin{enumerate}
\item[$\bullet$] if case (a) holds for $S'$, i.e., for some (possibly empty) words $u'_1$, $u'_2$ in letters $x_{l},\ldots,x_{t}$ and some $l\leq r\leq t$ we have that 
$u'_1(S) E_{i_l,i_{t+1}} u'_2(S) = E_{i_r, i_{r+1}}^{\de_r}$; thus  case (a) holds for $S$ for $u_1''=u'_1 u_1$ and $u_2''=u_2 u'_2$ by formula~\Ref{lemma_TDS_for_claim1_eq3}, since $u_1''u_2''$ is a miltilinear word;

\item[$\bullet$] if case (b) holds for $S'$, i.e., there are words $w_1,\ldots,w_r$ in letters $x_l,\ldots,x_t$ and (possibly empty) words $u'_1$, $u'_2$ in letters $x_l,\ldots,x_t$ such that $(w_1(S),\ldots,w_r(S))$ is a staircase sequence with the foundation  
$u'_1(S) E_{i_l,i_{t+1}} u'_2(S)$; thus  case (b) holds for $S$ for $w_1,\ldots,w_r$ and $u_1''=u'_1 u_1$, $u_2''=u_2 u'_2$ by formula~\Ref{lemma_TDS_for_claim1_eq3}, since $u_1''u_2''w_1\cdots w_r$ is a miltilinear word.
\end{enumerate} 

\noindent{}Therefore, Claim 1 is proven.
\end{enumerate}

\bigskip
\noindent{\bf Claim 2.} {\it Without loss of generality, we can assume that  $\de_l\neq \de_{l+1}$ for all $1\leq l<k$}.
\medskip

\noindent{}To prove Claim 2 we assume that for some $1\leq l<k$ one of the  following two cases holds.

\begin{enumerate}
\item[{\bf (i)}] Let $\de_l=\de_{l+1}=1$. Then we consider the word $u=x_l x_{l+1}$ and note that $u(S)=E_{i_l,i_{l+1}}E_{i_{l+1},i_{l+2}}=E_{i_l,i_{l+2}}$. We take away $E_{i_l,i_{l+1}}$, $E_{i_{l+1},i_{l+2}}$ from the three-diagonal sequence $S$ and add $E_{i_l,i_{l+2}}$ to $S$ instead of them. Denote by $S'$ the resulting three-diagonal sequence of length $k-1\geq1$. By induction, case (a) or case (b) is valid for $S'$. Obviously, the same case also holds for $S$.

\item[{\bf (ii)}] Let $\de_l=\de_{l+1}=\mathsf{T}$. Then we consider the word $u=x_{l+1} x_l$ and note that $u(S)=E_{i_{l+1},i_{l+2}}^{\mathsf{T}} 
 E_{i_{l},i_{l+1}}^{\mathsf{T}} =E_{i_{l},i_{l+2}}^{\mathsf{T}} $. We take away $E_{i_l,i_{l+1}}^{\mathsf{T}}$, $E_{i_{l+1},i_{l+2}}^{\mathsf{T}}$ from the three-diagonal sequence $S$ and add $E_{i_l,i_{l+2}}^{\mathsf{T}}$ to $S$ instead of them. Denote by $S'$ the resulting three-diagonal sequence of length $k-1\geq1$. By induction, case (a) or case (b) is valid for $S'$. Obviously, the same case also holds for $S$.
\end{enumerate}

The above reasoning implies that we can eliminate the case when $\de_l=\de_{l+1}$ for some $1\leq l<k$, i.e., Claim 2 is proven.
\end{proof}

\begin{cor}\label{cor_TD}
Assume that $S=\TD_{\un{i},\un{\de}}$ is a three-diagonal sequence of length $k\geq1$. Then there are  non-empty words $w_1,\ldots,w_r$, where $r\geq1$ is odd, in letters $x_1,\ldots,x_k$ and (possibly empty) words $u_1,u_2$ in letters $x_1,\ldots,x_k$ such that
\begin{eq}\label{eq_cor_form}
\rank(\Alt(w_1(S),\ldots,w_r(S))  + \al u_1(S) E_{i_1i_{k+1}} u_2(S)) = \left\{
\begin{array}{cc}
\frac{r-1}{2}, & \text{if }\al=-1 \\
\frac{r+1}{2}, & \text{if }\al\neq -1 \\
\end{array}
\right.
\end{eq}%
for all $\al\in\FF$. Moreover,
\begin{eq}\label{eq_cor_TD}
\deg(u_1) + \deg(u_2) + 1 \leq k+2 \;\text{ and }\; \deg(w_l)\leq k \text{ for all } 1\leq l\leq k.
\end{eq}
\end{cor}
\begin{proof} Assume that $\de_l=1$ for some $1\leq l\leq k$. Thus either case (a) or case (b) of Lemma~\ref{lemma_TD} holds. In what follows we use notations from Lemma~\ref{lemma_TD}. 

Assume that case (a) of Lemma~\ref{lemma_TD} holds. Thus
\begin{eq}\label{eq_lemma_TD_a}
\deg(u_1) + \deg(u_2)\leq k.
\end{eq}%
We take $w_1=x_r$. Since the rank from formula~\Ref{eq_cor_form}  is equal to zero in case $\al=-1$ and is equal to one otherwise, we obtain that equality~\Ref{eq_cor_form} holds. Inequality~\Ref{eq_lemma_TD_a} implies inequalities~\Ref{eq_cor_TD}.  

Assume that case (b) of Lemma~\ref{lemma_TD} holds. Thus
\begin{eq}\label{eq_lemma_TD_b}
\deg(u_1) + \deg(u_2)+\deg(w_1)+\cdots + \deg(w_r)\leq k.
\end{eq}%
It is easy to see that formula~\Ref{eq_cor_form} holds. Inequality~\Ref{eq_lemma_TD_b} implies inequalities~\Ref{eq_cor_TD}.  

Assume that $\de_1=\cdots=\de_k=\mathsf{T}$. Then we apply Lemma~\ref{lemma_TD_T} to $S$ and use the same reasoning as above in case (a), with the only difference that instead of inequality~\Ref{eq_lemma_TD_a} here we have inequality $\deg(u_1) + \deg(u_2)\leq k+1$.

%
\end{proof}

\section{Separation of orbits}\label{section_orbits}

\begin{lemma}\label{lemma_trivial}
For every $A=\diag(a_1,\ldots,a_n)\in M_n$ there are
$H_{\un{a},1},\ldots,H_{\un{a},n}\in \algX{1}$ of degrees less than $n$ such that $H_{\un{a},t}(A)=E_{tt}$ for all $1\leq t\leq n$, where $\un{a}=(a_1,\ldots,a_n)$.
\end{lemma}
\begin{proof}Let $1\leq t\leq n$ be fixed. 
Consider the linear matrix equation $A^{n-1}y_1 + A^{n-2}y_2 + \cdots + A y_{n-1} + I_n y_n = E_{tt}$ with respect to the variables  $y_1,\ldots,y_n\in\FF$. This equation is equivalent to the linear system of equations $B\un{y}=E_t$, where $\un{y}=(y_1,\ldots,y_n)^T$, $E_t=(0,\ldots,0,1,0,\ldots,0)^T\in\FF^n$ is a column-vector with the only non-zero element in the $t^{\rm th}$ position and $B$ is the Vandermonde matrix with $B_{ij}=a_i^{n-j}$ for all $1\leq i,j\leq n$. Since $a_1,\ldots,a_n$ are pairwise distinct, $B$ is invertible and the system has a solution. \end{proof}

\begin{example}\label{ex_lemma_trivial}
Assume $A=\diag(a_1,a_2)$ with $a_1\neq a_2$. Denote 
$$H_{\un{a},1}=\frac{1}{a_1-a_2} X_1 - \frac{a_2}{a_1-a_2} I_2 \;\;\text{ and }\;\; H_{\un{a},2}=-\frac{1}{a_1-a_2} X_1 + \frac{a_1}{a_1-a_2} I_2,$$
where $\un{a}=(a_1,a_2)$. Then $H_{\un{a},t}(A)=E_{tt}$ for $t=1,2$.

\end{example}


\begin{theo}\label{theo_main1}
The polynomial invariants
$$\Set_{\si}=\{\si_t(X_1) \,|\, 1\leq t\leq n\},$$
together with abstract invariants 
$$\Set_{\zeta}=\{\zeta(F) \,|\, F\in \algX{2} \text{ has degree less or equal to } 2n-1 \}$$
is a set that separates types of elements of $D_n$. 
\end{theo}
\begin{proof} Assume that $\un{A},\un{B}\in D_n$ and $f(\un{A})=f(\un{B})$ for all $f\in \Set_{\si} \cup \Set_{\zeta}$. To complete the proof it is enough to show that the types of $\un{A}$ and $\un{B}$ are the same. 

Since elements of $\Set_{\si}$ and $\Set_{\zeta}$ are abstract invariants, without loss of generality we can assume that $\un{A}$ and $\un{B}$ are canonical.  Since $\si_t(A_1)=\si_t(B_1)$ for all $1\leq t\leq n$, we obtain that $A_1=B_1=\diag(a_1,\ldots,a_n)$ for some $a_1<\ldots<a_n$. 

For all $1\leq t\leq n$ denote the element $H_{\un{a},t}\in \algX{1}$ from Lemma~\ref{lemma_trivial} by $H_t$, where $\un{a}=(a_1,\ldots,a_n)$.
Let graphs $G_{\!\calcA}$ and  $G_{\calcB}$, respectively, be the types of $\un{A}$ and $\un{B}$, respectively, where $\calcA$ and $\calcB$ are the corresponding canonical $\ast$-matrices, i.e., $A_2\in \calcA$ and $B_2\in \calcB$. Since  
\begin{eq}\label{eq_H}
H_i X_2 H_j(\un{A}) = E_{ii} A_2 E_{jj} = (A_2)_{ij} E_{ij}
\end{eq}
for all $1\leq i,j\leq n$, we obtain that 
\begin{eq}\label{eq_claim}
\zeta(H_i X_2 H_j)(\un{A})= \left\{
\begin{array}{rl}
0, & \text{if }\calcA_{ij} = 1\\
0, & \text{if } \calcA_{ij} = \ast \text{ and } (A_2)_{ij} \neq 0 \\
1, & \text{if } \calcA_{ij} = \ast \text{ and } (A_2)_{ij} =0 \\
1, & \text{if } \calcA_{ij}=0 \\
\end{array}
\right..
\end{eq}%
Similarly, we have that 
\begin{eq}\label{eq_claim2}
\zeta(H_i X_2 H_j)(\un{B})= \left\{
\begin{array}{rl}
0, & \text{if }\calcB_{ij} = 1\\
0, & \text{if } \calcB_{ij} = \ast \text{ and } (B_2)_{ij} \neq 0 \\
1, & \text{if } \calcB_{ij} = \ast \text{ and } (B_2)_{ij} =0 \\
1, & \text{if } \calcB_{ij}=0 \\
\end{array}
\right..
\end{eq}%
Define a matrix $C\in M_n$ as follows:
$$C_{ij}=\left\{
\begin{array}{cl}
0,& \text{ if } \calcA_{ij}=0 \text{ or } \calcB_{ij}=0  \\
1,& \text{ otherwise } \\
\end{array}
\right.$$
for all $1\leq i,j\leq n$. Note that $\zeta(H_i X_2 H_j)\in \Set_{\zeta}$. Thus $\zeta(H_i X_2 H_j)(\un{A})=\zeta(H_i X_2 H_j)(\un{B})$ and we can see that $C\in\calcA$ as well as $C\in\calcB$. Therefore, it follows from Lemma~\ref{lemma1} that $\calcA=\calcB$, i.e., $\un{A}$ and $\un{B}$ have one and the same type. 
\end{proof}

\begin{theo}\label{theo_main2}
The $\GL_n$-orbits on $D_n$ are separated by the set 
$$\Set_{\rank}=\{\rank(F) \,|\, F\in \algX{2} \text{ with } \deg(F)\leq (n+1)(2n-1) \}.$$
\end{theo}
\begin{proof}
Assume that $\un{A},\un{B}\in D_n$ and $f(\un{A})=f(\un{B})$ for all $f\in \Set_{\rank}$. To complete the proof we have to show that $\GL_n\cdot \un{A} = \GL_n\cdot \un{B}$. Since elements of $\Set_{\rank}$ are abstract invariants, without loss of generality we can assume that $\un{A}$ and $\un{B}$ are canonical.  Since $A_1,B_1$ are diagonal and $\rank(A_1 - \la I_n)=\rank(B_1 - \la I_n)$ for all $\la\in\FF$, we obtain that $A_1=B_1=\diag(a_1,\ldots,a_n)$ for some $a_1<\ldots<a_n$. Since for every $C\in M_n$ we have that $\zeta(C)=1$ if and only if $\rank(C)=0$, we can apply Theorem~\ref{theo_main1} to obtain that types of $\un{A}$ and $\un{B}$ are the same. As in the proof of Theorem~\ref{theo_main1},  the graph $G_{\!\calcA}$ is the type of $\un{A}$ and $\un{B}$, where $\calcA$ is the corresponding canonical $\ast$-matrix, i.e., $A_2,B_2\in \calcA$. For all $1\leq t\leq n$ denote the element $H_{\un{a},t}\in \algX{1}$ from Lemma~\ref{lemma_trivial} by $H_t$, where $\un{a}=(a_1,\ldots,a_n)$.

 To show that $A_2=B_2$ we consider some $1\leq i,j\leq n$. In this proof we establish that  $(A_2)_{ij}=(B_2)_{ij}$.

If $\calcA_{ij}\neq \ast$, then $\calcA_{ij}=0$ or $\calcA_{ij}=1$; in both cases we have $(A_2)_{ij}=(B_2)_{ij}$, since $A_2,B_2\in \calcA$. 

Assume $\calcA_{ij}=\ast$. Formula~\Ref{eq_claim} implies that  
$$\rank(H_i X_2 H_j)(\un{A}) = 0 \;\text{ if and only if }\; (A_2)_{ij}=0$$
and the same claim holds for $\rank(H_i X_2 H_j)(\un{B})$. Since 
$$\rank(H_i X_2 H_j)(\un{A})=\rank(H_i X_2 H_j)(\un{B}),$$ 
we obtain that $(A_2)_{ij}=0$ if and only if $(B_2)_{ij}=0$. Therefore, we assume that
\begin{eq}\label{eq_nonzero}
(A_2)_{ij}\neq 0 \text{ and } (B_2)_{ij}\neq0.
\end{eq}

Assume $i=j$. Consider the following element from $\Set_{\rank}$:
\begin{eq}\label{eq_def_etaii}
\eta_{\un{A}}^{i,i} := \rank\left((A_2)_{ii} I_n -  H_i X_2 H_i\right).
\end{eq}%
Since $\eta_{\un{A}}^{i,i}(\un{A})=\eta_{\un{A}}^{i,i}(\un{B})$, we obtain that $(A_2)_{ii}=(B_2)_{ii}$.

In the rest of the proof we assume that $i\neq j$. Since $\calcA_{ij}=\ast$, the definition of canonical $\ast$-matrix implies that in graph $G_{\!\calcA}$ there exists (a unique) undirected path $a$ from $v_i$ to $v_j$. We have $a=a_{i_1i_2}^{\de_1}a_{i_2i_3}^{\de_2}\cdots a_{i_ki_{k+1}}^{\de_k}$ for $i_1=i$, $i_{k+1}=j$, where  $\de_1,\ldots,\de_k\in\{1,\mathsf{T}\}$ and 
$$a_{ls}^{\de}=\left\{
\begin{array}{cc}
\text{the arrow } v_l\to v_s \text{ of } G_{\!\calcA}, & \text{ if } \de=1 \\
\text{the arrow } v_s\to v_l \text{ of } G_{\!\calcA}, & \text{ if } \de=\mathsf{T} \\
\end{array}
\right..
$$%
Consider a three-diagonal sequence  $S=\TD_{\un{i},\un{\de}}$ of length $k\geq1$. Corollary~\ref{cor_TD} implies that there are  non-empty words $w_1,\ldots,w_r$, where $r\geq1$, in letters $x_1,\ldots,x_k$ and (possibly empty) words $u_1,u_2$ in letters $x_1,\ldots,x_k$ such that 
$$\rank(\Alt(w_1(S),\ldots,w_r(S))  + \al u_1(S) E_{i_1i_{k+1}} u_2(S)) = \left\{
\begin{array}{cc}
\frac{r-1}{2}, & \text{if }\al=-1 \\
\frac{r+1}{2}, & \text{if }\al\neq -1 \\
\end{array}
\right.
$$ 
for all $\al\in\FF$. Denote $h_0=H_{i_1} X_2 H_{i_{k+1}}$ and for every $1\leq l\leq k$ denote 
$$h_l=\left\{
\begin{array}{cc}
H_{i_l} X_2 H_{i_{l+1}} & \text{ if } \de_l=1 \\
H_{i_{l+1}} X_2 H_{i_l} & \text{ if } \de_l=\mathsf{T} \\
\end{array}
\right..
$$%
Since $\calcA_{i_l,i_{l+1}} = 1$ in case $\de_l=1$ and  $\calcA_{i_{l+1},i_l} = 1$ in case $\de_l=\mathsf{T}$, formula~\Ref{eq_H} implies that 
\begin{eq}\label{eq_hA}
h_l(\un{A})=h_l(\un{B})=E_{i_l,i_{l+1}}^{\de_l}
\end{eq}%
Denote by $H=(h_1,\ldots,h_k)$ and for a word $w$ in letters $x_1,\ldots,x_k$ we write by $w(H)\in \algX{2}$ for the result of substitutions $x_1\to h_1,\ldots,x_k\to h_k$ in $w$. Note that it follows from formula~\Ref{eq_hA} that 
\begin{eq}\label{eq_wH}
w(H)(\un{A})=w(H)(\un{B})=w(S).
\end{eq}%

Consider the following $\GL_n$-invariant map from $M_n^2$ to $\ZZ$:
\begin{eq}\label{eq_def_eta}
\eta_{\un{A}}^{i,j} := \rank\left((A_2)_{ij}\Alt(w_1(H),\ldots, w_r(H)) - u_1(H) h_0 u_2(H)\right)
\end{eq}%
Formula~\Ref{eq_wH} implies that  
$$\eta_{\un{A}}^{i,j}(\un{A})=\rank\left(   (A_2)_{ij}
\Alt(w_1(S),\ldots,w_r(S))- (A_2)_{ij} u_1(S) E_{i_1,i_{k+1}} u_2(S) \right)=\frac{r-1}{2}.$$
Similarly, 
$$\eta_{\un{A}}^{i,j}(\un{B})=\rank\left(   (A_2)_{ij}
\Alt(w_1(S),\ldots,w_r(S))- (B_2)_{ij} u_1(S) E_{i_1,i_{k+1}} u_2(S) \right)$$
is equal to $(r-1)/2$ if and only if $(A_2)_{ij}=(B_2)_{ij}$. Therefore, $A_2=B_2$ in case $\eta_{\un{A}}^{i,j}\in \Set_{\rank}$. 

To complete the proof we have to establish that $\eta_{\un{A}}^{i,j}$ lies in $\Set_{\rank}$. Since $k$ is the length of the unique undirected path from $v_i$ to $v_j$ in the graph $G_{\!\calcA}$ with $n$ vertices, we obtain that $k\leq n-1$. Note that degrees of $h_0,h_1,\ldots,h_k\in \algX{2}$ are less than $2n$. Using restrictions on the degrees of words $w_1,\ldots,w_r$, $u_1,u_2$ from Corollary~\ref{cor_TD}  we obtain the required.
\end{proof}

The next corollary follows from Theorem~\ref{theo_main2} and Remark~\ref{remark_abs_inv2}.

\begin{cor}\label{cor_theo_main2}
The abstract invariants of width one:
$$\rank\nolimits_t(F), \text{ where } 0\leq t\leq n \text{ and }F\in \algX{2} \text{ with } \deg(F)\leq(n+1)(2n-1), $$
separate $\GL_n$-orbits on $D_n$.
\end{cor}

\begin{example}\label{ex_theo_main2} Consider $\un{a}=(a_1,\ldots,a_4)\in\FF^4$ with pairwise different entries and the canonical $\ast$-matrix 
$$\calcA = \left(
\begin{array}{cccc}
\ast & 0 & 0 & 0 \\
1 & \ast & 1 & 0 \\
\ast & \ast & \ast & 0 \\
1 & \ast & \ast & \ast \\
\end{array}
\right) \;\text{ with }\;\;
G_{\!\calcA}: v_4 \longrightarrow v_1 \longleftarrow v_2 \longrightarrow v_3.$$
Denote by $D_4^{\un{a},\calcA}$ the set of all pairs $\un{A}\in D_4$, where eigenvalues of $A_1$ are $\{a_1,a_2,a_3,a_4\}$ and the type of $\un{A}$ is $G_{\calcA}$.  For short, denote 
$$H_{ls}=H_{\un{a},l} X_2 H_{\un{a},s}\in\algX{2},$$
where $1\leq l,s\leq 4$ and the definition of $H_{\un{a},l}\in \algX{1}$ is given in Lemma~\ref{lemma_trivial}.  By Theorem~\ref{theo_main2}, the $\GL_4$-orbits on $D_4^{\un{a},\calcA}$ are separated by  $\Set_{\rank}$.  Moreover, $\GL_4$-orbits on $D_4^{\un{a},\calcA}$ are separated by the following proper subset $S$ of $\Set_{\rank}$:
$$\rank\left(\la I_4 -  H_{ii} \right), 1\leq i\leq 4, \;\; \rank\left(\la H_{21} -  H_{23} H_{31}\right), \;\; \rank\left(\la H_{23} -  H_{23} H_{32} H_{23}\right) $$
$$\rank\left(\la H_{41} -  H_{42} H_{21}\right),\;\; \rank\left(\la (H_{41} - H_{21} + H_{23}) -  H_{43}\right), \text{ where }\la\in\FF.$$

To prove this result we use the algorithm given in the proof of Theorem~\ref{theo_main2} for constructing the separating set $S$. Consider some $\un{A}\in D_4^{\un{a},\calcA}$ with $A_2=(\al_{ij})_{1\leq i,j\leq 4}$. For each $1\leq i,j\leq 4$ with $\calcA_{ij}=\ast$ we explicitly construct $\eta_{\un{A}}^{i,j}\in \Set_{\rank}$ that uniquely determines $\al_{ij}$. Note that $H_{ij}(\un{A})=\al_{ij} E_{ij}$ (see formula~\Ref{eq_H}).

\medskip
\noindent{}{\bf 1.} If $i=j$, then $\eta^{i,i}_{\un{A}}=\rank\left(\al_{ii} I_n -  H_{ii} \right)$.

\medskip
\noindent{}{\bf 2.} Let $(i,j)=(3,1)$. Then the undirected path from $v_3$ to $v_1$ is $a=a_{32}^{\mathsf{T}}a_{21}$ and $S=(E_{32}^{\mathsf{T}},E_{21})$. Thus equality~\Ref{eq_cor_form} holds for $w_1=x_2$, $u_1=x_1$, $u_2=1$, where $r=1$. Therefore,
$$\eta^{3,1}_{\un{A}}=\rank\left(\al_{31} H_{21} -  H_{23} H_{31}\right).$$

\medskip
\noindent{}{\bf 3.}  Let $(i,j)=(3,2)$. Then the undirected path from $v_3$ to $v_2$ is $a=a_{32}^{\mathsf{T}}$ and $S=(E_{32}^{\mathsf{T}})$. Thus equality~\Ref{eq_cor_form} holds for $w_1=x_1$, $u_1=u_2=x_1$, where $r=1$. Therefore,
$$\eta^{3,2}_{\un{A}}=\rank\left(\al_{32} H_{23} -  H_{23} H_{32} H_{23}\right).$$

\medskip
\noindent{}{\bf 4.}  Let $(i,j)=(4,2)$. Then the undirected path from $v_4$ to $v_2$ is $a=a_{41}a_{12}^{\mathsf{T}}$ and $S=(E_{41},E_{12}^{\mathsf{T}})$. Thus equality~\Ref{eq_cor_form} holds for $w_1=x_1$, $u_1=1$, $u_2=x_2$, where $r=1$. Therefore,
$$\eta^{4,2}_{\un{A}}=\rank\left(\al_{42} H_{41} -  H_{42} H_{21}\right).$$

\medskip
\noindent{}{\bf 5.}  Let $(i,j)=(4,3)$. Then the undirected path from $v_4$ to $v_2$ is $a=a_{41}a_{12}^{\mathsf{T}}a_{23}$ and $S=(E_{41},E_{12}^{\mathsf{T}},E_{23})$. Thus equality~\Ref{eq_cor_form} holds for $w_l=x_l$ for $1\leq l\leq 3$, $u_1=u_2=1$, where $r=3$. Therefore,
$$\eta^{4,3}_{\un{A}}=\rank\left(\al_{43} (H_{41} - H_{21} + H_{23}) -  H_{43}\right).$$
\end{example}

\section{Examples}\label{section_examples}

\begin{lemma}\label{lemma_ex1} The $\GL_n$-orbits on $M_n^m$ are not separated by abstract invariants of width one in case $n\geq 3$ and $m\geq2$. In particular,  Conjecture~\ref{conj_CH} does not hold in general for $n\geq3$ and $m\geq2$. 
\end{lemma}
\begin{proof}
Obviously, it is enough to prove this lemma for $n=3$ and $m=2$.

Consider $\un{A},\un{B}\in M_3^2$ given by $\un{A}=(E_{12},E_{13})$ and $\un{B}=(E_{12},E_{32})$.  For an arbitrary $F\in\algX{2}$ we have $F=\sum_i \al_i W_i + \al X_1 + \be X_2 + \ga I_3$, where $W_i$ is a product of more than one generic matrices and $\al,\al_i,\be,\ga \in \FF$. Since $A_1 A_2=A_2 A_1=0$ and $B_1 B_2=B_2 B_1=0$, we obtain that   
$$F(\un{A}) = \matrTri{\ga}{\al}{\be}{0}{\ga}{0}{0}{0}{\ga} \;\text{ and }\;
F(\un{B}) = \matrTri{\ga}{\al}{0}{0}{\ga}{0}{0}{\be}{\ga}. 
$$
Note that $\GL_3\cdot F(\un{A}) = \GL_3\cdot F(\un{B})$, since $F(\un{A}) =  F(\un{B})$ in case $\be=0$ and 
$$g\cdot F(\un{A}) = F(\un{B}) \;\text{ for }\;
g=\matrTri{\al}{1}{0}{0}{\al}{\be}{\be}{0}{0}\in \GL_3$$
in case $\be\neq0$. Therefore, $\un{A}$ and $\un{B}$ are not separated by abstract invariants of width one. 

Let us show that $\GL_3\cdot \un{A} \neq \GL_3\cdot \un{B}$. Assume that for some  $g\in \GL_3$ we have that $g\cdot \un{A}=\un{B}$. The condition $gA_1g^{-1}=B_1$ implies that  
$$g=\matrTri{a_1}{a_2}{a_3}{0}{a_1}{0}{0}{a_4}{a_5}$$
for some $a_1,\ldots,a_5\in\FF$. Thus it follows from $gA_2g^{-1}=B_2$ that $a_1=0$,  which contradicts to the condition that $\det(g)\neq0$.
\end{proof}

\begin{lemma}\label{lemma_ex} The set 
$$\si_t(F), \; \zeta(F),\; \text{ where } 1\leq t\leq n \;\text{ and }\; F\in\algX{2},$$
 does not separate $\GL_n$-orbits on $D_n$ in case $n\geq4$. In particular, the set $\Set_{\si}\cup \Set_{\zeta}$ defined in Theorem~\ref{theo_main1} does not separate $\GL_n$-orbits on $D_n$ in case $n\geq4$.
\end{lemma}
\begin{proof} Obviously, the statement of the lemma for $n>4$ follows from the case of $n=4$. 

Assume that $n=4$. Consider $A_1=B_1=\diag(a_1,\ldots,a_4)$, where $a_1,\ldots,a_4\in\FF$ are pairwise different,
$$A_2=\left(
\begin{array}{cccc}
0 & 0 & 0 & 0 \\
1 & 0 & 1 & 0 \\
0 & 0 & 0 & 0 \\
1 & 0 & \al & 0 \\
\end{array}
\right) \;\text{ and }\;
B_2=\left(
\begin{array}{cccc}
0 & 0 & 0 & 0 \\
1 & 0 & 1 & 0 \\
0 & 0 & 0 & 0 \\
1 & 0 & \be & 0 \\
\end{array}
\right)
$$
for non-zero $\al,\be\in\FF$ with $\al\neq \be$. Since $A_2,B_2\in\calcA$ for the canonical $\ast$-matrix from Example~\ref{ex_theo_main2}, namely, 
$$\calcA = \left(
\begin{array}{cccc}
\ast & 0 & 0 & 0 \\
1 & \ast & 1 & 0 \\
\ast & \ast & \ast & 0 \\
1 & \ast & \ast & \ast \\
\end{array}
\right) \;\text{ with }\;\;
G_{\!\calcA}: v_4 \longrightarrow v_1 \longleftarrow v_2 \longrightarrow v_3,$$
\noindent{}then $\un{A},\un{B}\in D_4$ have type $G_{\!\calcA}$. Theorem~\ref{theoFutorny} implies that $\GL_4\cdot\un{A}\neq \GL_4\cdot\un{B}$. 

Denote by $\calcL$ the set of all matrices
$$\left(
\begin{array}{cccc}
b_1 & 0 & 0 & 0 \\
b_2 & b_3 & b_4 & 0 \\
0 & 0 & b_5 & 0 \\
b_6 & 0  & b_7 & b_8 \\
\end{array}
\right)$$
for $b_1,\ldots,b_8\in\FF$. It is easy to see that  $\calcL$ is an algebra and $\calcL$ contains $A_1,A_2,B_1,B_2$. Denote by $\calcL^{(2)}$ the set of all pairs
$(L,N)\in \calcL^2$ such that
\begin{enumerate}
\item[$\bullet$] for every $1\leq i,j\leq 4$ with $(i,j)\neq (4,3)$ we have $L_{ij}=N_{ij}$; 

\item[$\bullet$] there exists $b\in\FF$ such that $L_{43}=\al b$ and $N_{43}=\be b$.
\end{enumerate}%
Obviously, $(A_1,B_1)$ and $(A_2,B_2)$ belong to $\calcL^{(2)}$. Note that if $(L_1,N_1)$ and $(L_2,N_2)$ belong to $\calcL^{(2)}$, then 
\begin{enumerate}
\item[$\bullet$] any linear combination of $(L_1,N_1)$ and $(L_2,N_2)$ lies in $\calcL^{(2)}$;

\item[$\bullet$] $(L_1L_2,N_1N_2)$ lies in $\calcL^{(2)}$.
\end{enumerate}
Therefore, for every $F\in\algX{2}$ we have that $(F(\un{A}),F(\un{B}))\in \calcL^{(2)}$. Note that if $(L,N)\in\calcL^{(2)}$, then $\si_t(L)=\si_t(N)$ for all $1\leq t\leq 4$ and $\zeta(L)=\zeta(N)$. Therefore, $f(\un{A})=f(\un{B})$ for every invariant $f$ from the formulation of this lemma. The required is proven.
\end{proof}

\end{document}